\documentclass[12pt]{amsart}
\usepackage{amssymb,amsmath}
\usepackage{geometry}    

\usepackage{mathrsfs}

\usepackage{vmargin}
\setmargrb{1in}{1in}{1in}{1in}

\newtheorem{theorem}{Theorem}[section]
\newtheorem{lemma}[theorem]{Lemma}
\newtheorem{proposition}[theorem]{Proposition}
\newtheorem{corollary}[theorem]{Corollary}
\theoremstyle{definition}
\newtheorem{definition}[theorem]{Definition}
\newtheorem{remark}[theorem]{Remark}

\newtheorem{conjecture}[theorem]{Conjecture}
\newtheorem{problem}[theorem]{Problem}

\begin{document}

\title[Rational approximation and a new exponent]
{Rational approximation to algebraic varieties and a new exponent of simultaneous approximation}

\author{Johannes Schleischitz}

\begin{abstract}
This paper deals with two main topics related to Diophantine approximation. Firstly,
we show that if a point on an algebraic variety is approximable by rational vectors to 
a sufficiently large degree, the approximating vectors must lie in the topological 
closure of the rational points on the variety. In many interesting cases, in particular
if the set of rational points on the variety is finite, this closure does not exceed
the set of rational points on the variety itself. 
This result enables easier proofs of several known results as special cases.
The proof can be generalized in some way
and encourages to define a new exponent of simultaneous approximation. The second part
of the paper is devoted to the study of this exponent. 
\end{abstract}

\maketitle

{\footnotesize{Supported by FWF grant P24828} \\
Institute of Mathematics, Department of Integrative Biology, BOKU Wien, 1180, Vienna, Austria.
{\tt johannes.schleischitz@univie.ac.at} \\

Math subject classification: 11J13, 11J82, 11J83 \\
Key words: exponents of Diophantine approximation, rational points on varieties, continued fractions}

\vspace{4mm}

\section{Introduction} \label{intro}

In this paper we study certain aspects concerning 
the simultaneous approximation of vectors $\underline{\zeta}\in{\mathbb{R}^{k}}$
by rational vectors. In the classical setting of simultaneous approximation
the approximating rational vectors are
of the form $(p_{1}/q,\ldots,p_{k}/q)\in{\mathbb{Q}^{k}}$ and the maximum of $\vert \zeta_{i}-p_{i}/q\vert$
is compared with the size of (large) $q$. 
In Sections~\ref{intro},\ref{sektion2} we stick to this classical setting and derive a new result concerning
very well approximable points on varieties that generalizes several results that have been established. 
This main result has a natural extension to the case where the denominators of the rational approximations 
may differ. Motivated by this we will introduce a new exponent of simultaneous approximation
in Section~\ref{segdsion2} and study its properties.

We first introduce some notation.

\begin{definition} \label{xydef}
Let $k\geq 1$ be an integer.
For a function $\psi:\mathbb{R}\to\mathbb{R}$ let $\mathscr{H}^{k}_{\psi}\subseteq \mathbb{R}^{k}$ be
the set of points $\underline{\zeta}=(\zeta_{1},\ldots,\zeta_{k})$ approximable to degree $\psi$, that is such that
\[
\max_{1\leq j\leq k} \vert x\zeta_{j}-y_{j}\vert \leq \psi(x)
\]
has a solution $(x,y_{1},\ldots,y_{k})\in{\mathbb{Z}^{k+1}}$ for arbitrarily large values of $x$.
If $\psi(x)=x^{-\mu}$ for $\mu>0$, we will also write $\mathscr{H}^{k}_{\mu}$ for $\mathscr{H}^{k}_{\psi}$
and refer to $\underline{\zeta}$ as approximable to degree $\mu$.
\end{definition}

Dirichlet's Theorem can be formulated in the way that $\mathscr{H}^{k}_{1/k}$ 
equals the entire space $\mathbb{R}^{k}$.
Thus only functions $\psi(x)\leq x^{-1/k}$ for  large $x$ resp. parameters $\mu>1/k$ are of interest.
Furthermore it is known thanks to Khintchine~\cite{khint}                
that the set $\mathscr{H}^{k}_{1/k+\delta}$ for any fixed
$\delta>0$ has $k$-dimensional Lebesgue measure $0$. 
On the other hand, the set $\cup_{\delta>0} \mathscr{H}^{k}_{1/k+\delta}$
often referred to as (simultaneously) very well approximable vectors, 
has full Hausdorff dimension $k$, see~\cite{jarnik}.                
As usual denote by $\Vert. \Vert$ the distance of a real number to the nearest integer.
Next we define constants closely related
to the sets $\mathscr{H}^{k}_{\mu}$ that have been intensely studied.

\begin{definition}
Let $k\geq 1$ be an integer. For $\underline{\zeta}=(\zeta_{1},\ldots,\zeta_{k})\in{\mathbb{R}^{k}}$
let $\omega_{k}(\underline{\zeta})$ be the 
exponent of classical $k$-dimensional rational approximation, i.e. the supremum of $\nu>0$
such that
\[
\max_{1\leq j\leq k} \Vert x\zeta_{j}\Vert \leq x^{-\nu}
\]
has infinitely many integral solutions $x$.
Similarly, let $\widehat{\omega}_{k}(\zeta)$ be the supremum of $\mu$ such that the system
\[
0<x\leq X, \qquad \max_{1\leq j\leq k} \Vert x\zeta_{j}\Vert \leq X^{-\mu}
\]
has an integral solutions $x$ for every large parameter $X$.
\end{definition}

The sets $\mathscr{H}^{k}_{\mu}$ coincide
with the sets $\{\underline{\zeta}\in{\mathbb{R}^{k}}: \omega_{k}(\underline{\zeta})\geq \mu\}$
for every $\mu>0$, respectively. 
For the special case of $\underline{\zeta}$ successive powers of a real number 
this leads to the quantities $\lambda_{k}, \widehat{\lambda}_{k}$ defined by
Bugeaud and Laurent~\cite{buglau}.                                    

\begin{definition}
Let $k\geq 1$. For $\zeta\in{\mathbb{R}}$ define $\lambda_{k}(\zeta)$ as the supremum of real $\mu$
such that
\[
\max_{1\leq j\leq k} \Vert x\zeta^{j}\Vert \leq x^{-\mu}
\]
has arbitrarily large solutions $x$. 
Similarly, let $\widehat{\lambda}_{k}(\zeta)$ be the supremum of $\mu$ such that the system
\[
0<x\leq X, \qquad \max_{1\leq j\leq k} \Vert x\zeta^{j}\Vert \leq X^{-\mu}
\]
has an integral solutions $x$ for every large parameter $X$.
\end{definition}

In particular the classic one-dimensional approximation constants $\lambda_{1}(\zeta)$ for $\zeta\in{\mathbb{R}}$ 
is defined as the supremum of real $\mu$ such that $\Vert x\zeta\Vert\leq x^{-\mu}$ has 
arbitrarily large solutions $x$. For $k=1$ obviously $\omega_{1}(\zeta)=\lambda_{1}(\zeta)$
and consequently the sets $\mathscr{H}^{1}_{\mu}$ coincide
with the set $\{\zeta\in{\mathbb{R}}: \lambda_{1}(\zeta)\geq \mu\}$. 
Clearly $1/k\leq \widehat{\lambda}_{k}(\zeta)\leq \lambda_{k}(\zeta)$ for all $k$ and $\zeta$ such as
\[
\lambda_{1}(\zeta)\geq \lambda_{2}(\zeta)\geq \cdots, \qquad
\widehat{\lambda}_{1}(\zeta)\geq \widehat{\lambda}_{2}(\zeta)\geq \cdots 
\]
for every $\zeta$.
Moreover, we have $\widehat{\lambda}_{1}(\zeta)=1$ for every irrational $\zeta$ and 
$\lambda_{k}(\zeta)=1/k$ for almost all $\zeta$ in the sense of Lebesgue measure~\cite{sprindzuk}.
For further results concerning the spectrum of the exponents see for example~\cite{bug},~\cite{buglau},~\cite{schlei}.

Finally we introduce the absolute degree of a polynomial.

\begin{definition}
For a monomial $M:=aX_{1}^{j_{1}}\cdots X_{k}^{j_{k}}$ with $a\in{\mathbb{Q}\setminus\{0\}}$ 
let $j_{1}+\cdots+j_{k}$ be the total degree of $M$. 
For $P\in{\mathbb{Q}[X_{1},\ldots,X_{k}]}$ define the {\em absolute degree} of $P$ 
as the maximum of the total degrees of the monomials involved in $P$.
\end{definition}

\section{A result on approximation to varieties} \label{sektion2}

Theorem~\ref{parameter} is the main result of this section. Its proof is not difficult and based
on the fact that if a polynomial with rational coefficients of absolute degree $r$ 
does not vanish at some point $(y_{1}/x,\ldots,y_{k}/x)$ then the evaluation is bounded below essentially by $x^{-r}$.
We partly state it because in view of Theorem~\ref{chithm} below it will help to motivate the
new exponent we will introduce in Section~\ref{segdsion2}.

\begin{theorem} \label{parameter}
Let $P\in{\mathbb{Q}[X_{1},\ldots,X_{k}]}$ of absolute degree $r$
and $V$ be the variety defined by 
\[
V=\{(X_{1},X_{2},\ldots,X_{k})\in{\mathbb{R}^{k}}: \quad
P(X_{1},X_{2},\ldots,X_{k})=0\}.
\]
Denote $\mathscr{T}:=V\cap \mathbb{Q}^{k}$ the rational points on $V$. 
Let $\psi:\mathbb{R}\to\mathbb{R}$ be any function with the property
$\psi(t)=o(t^{-r+1})$ as $t\to\infty$. 
Then $\mathscr{T}\subseteq \mathscr{H}^{k}_{\psi}\cap V\subseteq \overline{\mathscr{T}}$,
where $\overline{\mathscr{T}}$ denotes the topological closure of $\mathscr{T}$ with respect to
the usual Euclidean metric.
\end{theorem}

\begin{proof}
Clearly we may assume $P\in{\mathbb{Z}[X_{1},\ldots,X_{k}]}$. It also obvious that
$\mathscr{T}\subseteq \mathscr{H}^{k}_{\psi}\cap V$ for an arbitrary function $\psi$, since given
$(p_{1}/q,\ldots,p_{k}/q)\in{\mathscr{T}}$
it suffices to take $(x,y_{1},\ldots,y_{k})=(Mq,Mp_{1},\ldots,Mp_{k})$ the integral multiples of the vector 
($M\in\{1,2,\ldots\}$) in Definition~\ref{xydef}. 
We must prove that $\mathscr{H}^{k}_{\psi}\cap V\subseteq \overline{\mathscr{T}}$ for $\psi(t)=o(t^{-r+1})$.

Let $\underline{\zeta}=(\zeta_{1},\ldots,\zeta_{k})\in{V\setminus \overline{\mathscr{T}}}$. 
We have to show $\underline{\zeta}\notin{\mathscr{H}^{k}_{\psi}}$. Assume $\underline{\zeta}\in{\mathscr{H}^{k}_{\psi}}$.
By definition we have
\[
\left\vert \zeta_{j}-\frac{y_{j}}{x}\right\vert \leq \psi(x)x^{-1}, \qquad 1\leq j\leq k
\]
for arbitrarily large $x$. Hence we can write $\zeta_{j}=y_{j}/x+\epsilon_{j}$ with 
$\vert \epsilon_{j}\vert\leq \psi(x)x^{-1}$ for $1\leq j\leq k$. 
Since $\underline{\zeta}\notin{\overline{\mathscr{T}}}$, there exists some open neighborhood $U\ni{x}$ of $x$
such that $U\cap \mathscr{T}=\emptyset$, or in other words there is no rational point in $U\cap V$.
Observe that $P$ is $C^{\infty}$ on $\mathbb{R}^{k}$, thus in $U$
the partial derivatives $P_{x_{1}},\ldots,P_{x_{k}}$
are uniformly bounded by some constant $C$ in absolute value. We may assume
$x$ to be large enough that $(y_{1}/x,\ldots,y_{k}/x)\in{U}$.
With repeated use of (one-dimensional) Taylor Theorem parallel to the coordinate axes we obtain
\begin{equation} \label{eq:jippy}
0=P(\underline{\zeta})=P\left(\frac{y_{1}}{x}+\epsilon_{1},\ldots,\frac{y_{k}}{x}+\epsilon_{k}\right)=
P\left(\frac{y_{1}}{x},\ldots,\frac{y_{k}}{x}\right)+ 
\epsilon_{1} P_{x_{1}}(t_{1})+ \cdots +\epsilon_{k} P_{x_{k}}(t_{k})
\end{equation}
where $t_{j}\in{U}$. Thus
\begin{equation} \label{eq:yeah}
\left\vert P(\underline{\zeta})-P\left(\frac{y_{1}}{x},\ldots,\frac{y_{k}}{x}\right)\right\vert\leq kC\cdot \max \vert \epsilon_{j}\vert
\leq kC\cdot \psi(x)x^{-1}.
\end{equation}
Since $(y_{1}/x,\ldots,y_{k}/x)\in{V\cap U}$ which has empty intersection with $\mathbb{Q}^{k}$ 
we derive 
\[
P(y_{1}/x,\ldots,y_{k}/x)\neq 0.
\]
Thus and since $P\in{\mathbb{Z}[X_{1},\ldots,X_{k}]}$ has absolute degree $r$ we obtain
$\vert P(y_{1}/x,\ldots,y_{k}/x)\vert \geq x^{-r}$. Hence and since $\psi(t)=o(t^{-r+1})$, for large $x$ 
from \eqref{eq:jippy} and \eqref{eq:yeah} we infer
\[
\vert P(\underline{\zeta})\vert \geq 
\left\vert P\left(\frac{y_{1}}{x},\ldots,\frac{y_{k}}{x}\right)\right\vert-\left\vert P(\underline{\zeta})-
P\left(\frac{y_{1}}{x},\ldots,\frac{y_{k}}{x}\right)\right\vert
\geq x^{-r}-kC\cdot x^{-1}\psi(x)\geq 
\frac{1}{2}x^{-r}. 
\]
This contradicts $P(\underline{\zeta})=0$. Hence indeed $\underline{\zeta}\notin{\mathscr{H}^{k}_{\psi}}$
and the proof is finished.
\end{proof}

The theorem in particular applies if $\mathscr{T}$ is finite. 

\begin{corollary} \label{spezfall}
With the definitions and assumptions of Theorem~\ref{parameter} assume that the set $\mathscr{T}$
of rational points on $V$ is finite. Then $\mathscr{H}^{k}_{\psi}\cap V=\mathscr{T}$.
\end{corollary}

Corollary~\ref{spezfall} contains 
various known results as special cases. For example the Fermat curve defined as the set of zeros of
$P(X,Y)=X^{k}+Y^{k}-1$ has only possibly the trivial points $\{(\pm 1,0),(0,\pm 1)\}$ 
approximable to degree greater $k-1$, which was established by Bernik and Dodson~\cite[p. 94]{berdod}.         
Corollary~\ref{spezfall} also implies one of the two claims
of the main result of~\cite[Theorem~1.1]{drutu} by Dru\c{t}u. Concretely it asserts 
that for a quadratic form $\mathcal{Q}$ in arbitrary many variables, if there are 
no rational points on the variety defined by $\mathcal{Q}(\underline{X})-1=0$, then
there are no points on this variety approximable to degree greater than $1$. 
In fact Theorem~\ref{parameter} generalizes~\cite[Lemma~4.1.1]{drutu} which readily implied this claim.  
However, it should be pointed out that the main and much more technical 
result of~\cite[Theorem~1.1]{drutu} is the other claim, which provides a formula for
the Hausdorff dimension for the variety as above in the case that it contains rational points.
Observe also that Corollary~\ref{spezfall} implies that an elliptic curve
of rank $0$ contains only finitely many points approximable to degree larger than $3$ by rational vectors.
We want to add that a very similar result was proved for very well approximable points
on surfaces parametrized by polynomials with rational coefficients, see~\cite[Lemma~1]{bu}. 

The case that $\mathscr{T}$ in Theorem~\ref{parameter}
is infinite but consists solely of isolated rational points that may have some non-rational
limit point on $V$ (observe $V$ is closed) is of interest.
The question arises how large the set $\overline{\mathscr{T}}\setminus \mathscr{T}$
of such limit points can be, for example in sense of Hausdorff measure. It is already not obvious how to find
an algebraic variety where $\mathscr{T}$ is infinite and consists solely of isolated points. 

\section{A new exponent of simultaneous approximation} \label{segdsion2}

The proof of Theorem~\ref{parameter} can be extended in some way to a similar 
Diophantine approximation problem that seems so far unstudied in the literature. 
We first define the new exponent of simultaneous approximation below and derive some propoerties, 
and will return to the connection with Section~\ref{sektion2} in Theorem~\ref{chithm}.

For a real function $\psi(t)$ that tends to $0$ as $t\to\infty$ let $\mathscr{Z}^{k}_{\psi}$ be the set of
$\underline{\zeta}=(\zeta_{1},\ldots,\zeta_{k})\in{\mathbb{R}^{k}}$ such that the system
\begin{equation} \label{eq:referee}
0<\min_{1\leq j\leq k} \vert x_{j}\vert\leq \max_{1\leq j\leq k} \vert x_{j}\vert \leq X, \qquad 
\max_{1\leq j\leq k} \Vert x_{j}\zeta_{j}\Vert\leq \psi(X)
\end{equation}
has a solution $(x_{1},\ldots,x_{k})\in{\mathbb{Z}^{k}}$ for arbitrarily large $X$. 
Moreover write 
$\mathscr{Z}^{k}_{\nu}$ instead of $\mathscr{Z}^{k}_{\psi}$ when $\psi(t)=t^{-\nu}$ with a parameter $\nu>0$.
Further denote by $\chi_{k}(\underline{\zeta})$ the supremum of exponents $\nu$ for which
$\underline{\zeta}\in{\mathscr{Z}^{k}_{\nu}}$, such that
\[
\mathscr{Z}^{k}_{\nu}= \{\underline{\zeta}\in{\mathbb{R}^{k}}: \chi_{k}(\underline{\zeta})\geq \nu\}.
\]
Obviously $\mathscr{Z}^{k}_{\psi}\supseteq \mathscr{H}^{k}_{\psi}$ for all $k\geq 1$ and
any $\underline{\zeta}\in{\mathbb{R}^{k}}$
for any function $\psi$, with equality
if $k=1$. In particular $\chi_{k}(\underline{\zeta})\geq \omega_{k}(\underline{\zeta})$ for 
all $k\geq 1$ and all $\underline{\zeta}\in{\mathbb{R}^{k}}$. 
Moreover $\mathscr{Z}^{k}_{1}=\mathbb{R}^{k}$ by the uniform version of Dirichlet's~Theorem applied
to any single $\zeta_{j}$. 
Furthermore the $k$-dimensional exponent is trivially bounded above by the
minimum of the one-dimensional constants $\lambda_{1}(\zeta_{j})$. As stated in Section~\ref{intro}
each of these single exponents equals $1$ also for almost all $\zeta\in{\mathbb{R}}$ 
in terms of Lebesgue measure.                             
Hence for almost
all $\underline{\zeta}\in{\mathbb{R}^{k}}$ we have $\chi_{k}(\underline{\zeta})=1$. 
Moreover
by Roth's Theorem $\chi_{k}(\underline{\zeta})=1$ if there is at least one irrational algebraic
element among the $\zeta_{j}$.

We can reformulate the above observations by the formula
\begin{equation}  \label{eq:nigeria}
\max\{1,\omega_{k}(\underline{\zeta})\}\leq \chi_{k}(\underline{\zeta})\leq \min_{1\leq j\leq k} \lambda_{1}(\zeta_{j}).
\end{equation}
Recall the one-dimensional constants $\lambda_{1}(\zeta)$ are determined by the continued fraction expansion
of $\zeta$.
Roughly speaking, the exponent $\chi_{k}$ somehow measures the distances of denominators of those convergents 
$p_{j}/q_{j}$, which lead to 
very good approximation $\vert p_{j}/q_{j}-\zeta_{j}\vert$ of the $\zeta_{j}$, compared to the single
$q_{j}$. The situation is different for the exponents $\omega_{k}$, where denominators of continued fractions
of single $\zeta_{j}$ lead to a large exponent $\omega_{k}$ only if 
their lowest common multiple is small compared to the smallest single $q_{j}$. 
Roughly speaking the exponents $\chi_{k}$ measure something in between the separate one-dimensional best approximations 
$\lambda_{1}$ of the single $\zeta_{j}$
and the classical simultaneous approximation constants $\omega_{k}$.
Another relation between $\chi_{k}$ and $\omega_{k}$ is given by the following easy lemma
where this phenomenon becomes apparent.

\begin{lemma} \label{verbindung}
Let $k\geq 1$ and $\underline{\zeta}\in{\mathbb{R}^{k}}$. We have
\[
\omega_{k}(\underline{\zeta})\geq \frac{\chi_{k}(\underline{\zeta})-k+1}{k}.
\]
\end{lemma}

\begin{proof}
Assume the system
\[
0<\max_{1\leq j\leq k} \vert q_{j}\vert \leq Q, \qquad \max_{1\leq j\leq k} \Vert q_{j}\zeta_{j}\Vert\leq Q^{-\nu},
\]
is satisfied. Then $0<q_{1}\cdots q_{k}\leq Q^{k}$ and
\[
\Vert q_{1}q_{2}\cdots q_{k}\zeta_{j}\Vert \leq (q_{1}q_{2}\cdots q_{j-1}q_{j+1}\cdots q_{k})\Vert q_{j}\zeta_{j}\Vert
\leq Q^{k-1-\nu}=(Q^{k})^{-(\nu-k+1)/k}, \quad 1\leq j\leq k.
\]
The claim follows since we may let $\nu$ arbitrarily close to $\chi_{k}(\underline{\zeta})$.
\end{proof}

Uniform exponents can be defined similarly to the classical simultaneous Diophantine approximation constants,
but since Dirichlet's~Theorem is uniform in the parameter $Q$ again (for irrational $\zeta_{j}$)
\[
1=\max\{1,\widehat{\omega}_{k}(\underline{\zeta})\}\leq 
\widehat{\chi}_{k}(\underline{\zeta})\leq \min_{1\leq j\leq k} \widehat{\lambda}_{1}(\zeta_{j})=1,
\]
and hence 
\[
\widehat{\chi}_{k}(\underline{\zeta})=1
\]
for all $\underline{\zeta}\notin{\mathbb{Q}^{k}}$ (for $\zeta\in{\mathbb{Q}}$ we have
$\widehat{\lambda}_{1}(\zeta_{j})=\infty$). We formulate some questions concerning the
constants $\chi_{k}$ similar to well-known (partially answered)
problems for the classic exponents $\omega_{k}, \lambda_{k}$, see for example 
\cite[Problem~1-3]{bug}. By the spectrum of
$\chi_{k}$ we will mean the set 
$\{ \chi_{k}(\underline{\zeta}): \underline{\zeta}\in{T_{k}}\}\subseteq \mathbb{R}$
of values taken by $\chi_{k}$ in the set $T_{k}\subseteq \mathbb{R}^{k}$ of $\underline{\zeta}\in{\mathbb{R}^{k}}$ which are linearly independent together with $\{1\}$ over $\mathbb{Q}$.

\begin{problem} \label{que1}
Is the spectrum of $\chi_{k}$ equal to $[1,\infty]$?
Find explicit constructions of $\underline{\zeta}\in{\mathbb{R}^{k}}$ with prescribed values of 
$\chi_{k}(\underline{\zeta})$.
\end{problem}

\begin{problem} \label{que2}
Metric theory: For $\lambda\in{[1,\infty]}$ determine the Hausdorff dimensions of the sets
\[
\dim(\{\underline{\zeta}\in{\mathbb{R}^{k}}: \chi_{k}(\underline{\zeta})=\lambda\}),
\qquad \dim(\{\underline{\zeta}\in{\mathbb{R}^{k}}: \chi_{k}(\underline{\zeta})\geq \lambda\}).
\]

\end{problem}

\begin{problem} \label{que3}
What about Problems~\ref{que1}, \ref{que2} for the 
restriction of $\underline{\zeta}$
to certain manifolds in $\mathbb{R}^{k}$? In particular the Veronese curve which consists of the vectors 
$\underline{\zeta}=(\zeta,\zeta^{2},\ldots,\zeta^{k})$ 
for $\zeta\in{\mathbb{R}}$. 
\end{problem}  

Concerning Problem~\ref{que2}, we point out that the estimates 
\begin{equation} \label{eq:lehrer}
\frac{k+1}{1+\lambda}\leq \dim(\{\underline{\zeta}\in{\mathbb{R}^{k}}: \chi_{k}(\underline{\zeta})\geq \lambda\})
\leq \frac{k+1}{1+\frac{\lambda-k+1}{k}}=\frac{k(k+1)}{1+\lambda}
\end{equation}
hold, where the right inequality is non-trivial only for $\lambda>k$. Indeed
Jarn\'ik~\cite{jarnik} proved 
\[
\frac{k+1}{1+\lambda}=\dim(\{\underline{\zeta}\in{\mathbb{R}^{k}}: \omega_{k}(\underline{\zeta})\geq \lambda\})
=\dim(\{\underline{\zeta}\in{\mathbb{R}^{k}}: \omega_{k}(\underline{\zeta})=\lambda\})
\]
for $\lambda\in{[1/k,\infty]}$, which in combination with $\chi_{k}(\underline{\zeta})\geq \omega_{k}(\underline{\zeta})$
and Lemma~\ref{verbindung} respectively proves the inequalities in \eqref{eq:lehrer} respectively.

Concerning Problem~\ref{que3} for varieties, a slight modification of the proof of Theorem~\ref{parameter}
shows the following.

\begin{theorem} \label{chithm}
Let $P\in{\mathbb{Q}[X_{1},\ldots,X_{k}]}$ of absolute degree $r$
and $V$ be the variety defined by 
\[
V=\{(X_{1},X_{2},\ldots,X_{k})\in{\mathbb{R}^{k}}: \quad
P(X_{1},X_{2},\ldots,X_{k})=0\}.
\]
Denote $\mathscr{T}:=V\cap \mathbb{Q}^{k}$ the rational points on $V$. 
Let $\psi:\mathbb{R}\to\mathbb{R}$ be any function with the property
$\psi(X)=o(X^{-kr+1})$ as $X\to\infty$. 
Then $\mathscr{T}\subseteq \mathscr{Z}^{k}_{\psi}\cap V\subseteq \overline{\mathscr{T}}$.
\end{theorem} 

\begin{proof}[Proof of Theorem~\ref{chithm}]
Proceed precisely as in the proof of Theorem~\ref{parameter}, and notice that 
for general fractions $\underline{z}:=(p_{1}/q_{1},\ldots,p_{k}/q_{k})$ we still have the lower bound
$\vert P(\underline{z})\vert \geq q_{1}^{-r}q_{2}^{-r}\cdots q_{k}^{-r}\geq Q^{-kr}$.
\end{proof} 

\begin{remark} \label{dieremb}
The proof shows that for the large class of varieties
the exponent $kr-1$ can be readily improved. 
This is the case if the polynomial does not contain {\em all} monomials
$a_{1}X_{1}^{r},a_{2}X_{2}^{r},\cdots,a_{k}X_{k}^{r}$ with non-zero coefficients $a_{i}\neq 0$.
More precisely the condition $\psi(x)=o(x^{-r+1})$, with $r:=\sum_{j=1}^{k} r_{j}\leq kr$
where $r_{j}\leq r$ is the degree of $P(X_{1},\ldots,X_{k})$ in the variable $X_{j}$, suffices
to obtain the result of Theorem~\ref{chithm}.
In particular if $P$ is of the form
$P(X_{1},\ldots,X_{k})=X_{1}^{r_{1}}X_{2}^{r_{2}}\cdots X_{k}^{r_{k}}-l_{1}/l_{2}$ for $l_{1}/l_{2}\in{\mathbb{Q}}$,
then $\psi(x)=o(x^{-r+1})$ is sufficient.
More generally this applies for
$P(X_{1},\ldots,X_{k})=(p/q)X_{1}^{r_{1}}X_{2}^{r_{2}}\cdots X_{k}^{r_{k}}+Q(X_{1},\ldots,X_{k})$
for $p/q\in{\mathbb{Q}}$ and any $Q\in{\mathbb{Q}[X_{1},\ldots,X_{k}]}$ of degree 
at most $r_{j}$ in the variable $X_{j}$ for $1\leq j\leq k$. 
\end{remark}

We want to point out some consequences and interpretations of Theorem~\ref{chithm}, which 
also aim to shed more light on the meaning of the exponent $\chi_{k}$ in general. 
Recall a Liouville number
is an irrational real (and thus transcendental by Liouville's~Theorem) 
number that satisfies $\lambda_{1}(\zeta)=\infty$.
It is shown in~\cite{kumar} that for any 
countable set of continuous strictly monotonic functions $f_{i}:A\to B$
with $A,B$ non-empty intervals of $\mathbb{R}$, there are uncountably many Liouville numbers $\zeta\in{A}$ such that 
$f_{i}(\zeta)$ is again a Liouville number for all $i$. See also~\cite{rieger},~\cite{schwarz}.
Let $C$ be any curve in $\mathbb{R}^{k}$ for arbitrary $k$ defined by algebraic equations. Then $C$ 
can be almost everywhere locally parametrized by such functions $f_{0}=\rm{id}$,$f_{1},\ldots,f_{k-1}$,
in other words any $(\zeta_{1},\ldots,\zeta_{k})\in{C}$ can be written
$\zeta_{i+1}=f_{i}(\zeta)$ for $0\leq i\leq k-1$.
Hence there are uncountably many Liouville points on the curve, by which we mean that
every coordinate is a Liouville number. On the other hand, if $C$ is a rational variety that contains no rational
point, by Theorem~\ref{chithm} there are also no points simultaneously approximable to a sufficiently
large finite degree in the sense of large $\chi_{k}$ (of course also not for $\omega_{k}$). 
This emphasizes that on algebraic curves there is a huge
difference between the minimum of the one-dimensional classical constants $\lambda_{1}(\zeta_{j})$ 
and the constants $\chi_{k}(\underline{\zeta})$. For $0\leq i\leq k-1$ 
denote by $(p_{n,i}/q_{n,i})_{n\geq 1}$ the sequence of convergents of $f_{i}(\zeta)$.
Then the above result means that for the Liouville numbers 
$\zeta,f_{1}(\zeta),\ldots,f_{k-1}(\zeta)$
in the parametrization there do not exist infinitely many convergents $p_{.,0}/q_{.,0},\ldots,p_{.,k-1}/q_{.,k-1}$ 
whose denominators $q_{.,i}, 0\leq i\leq k-1$ are all of ''similar'' largeness. The analogous phenomenon
holds for all algebraic surfaces of dimension larger than one as well. Indeed, if the dimension of the variety is locally $k$,
then we can write the variety locally as 
$(\zeta_{1},\ldots,\zeta_{k},\psi_{1}(\underline{\zeta}),\ldots,\psi_{r}(\underline{\zeta}))$ 
with $\underline{\zeta}=(\zeta_{1},\ldots,\zeta_{k})$ and $C^{\infty}$ functions $\psi_{j}$
in some open $U$ subset of $\mathbb{R}^{k}$. 
We fix the first $k-1$ coordinates as Liouville
numbers in some open subset of $\mathbb{R}^{k-1}$ (i.e. we pick Liouville numbers
in the open projection set $V\subseteq U$ of $U$ to the first $k-1$ coordinates)
and the analogue result follows from the one-dimensional case.   

Concerning the spectrum of the quantities $\chi_{k}(\underline{\zeta})$ the next theorem is rather satisfactory. 

\begin{theorem} \label{lamblemm}
Let $k\geq 2$ an integer and  $\lambda_{1},\lambda_{2},\ldots,\lambda_{k},w$ real numbers that 
satisfy $1\leq w\leq \min_{1\leq j\leq k} \lambda_{j}$. Then 
there exist uncountably many
vectors $(\zeta_{1},\zeta_{2},\ldots,\zeta_{k})\in{\mathbb{R}^{k}}$ that are $\mathbb{Q}$-linearly independent
together with $\{1\}$ and such that 
$\lambda_{1}(\zeta_{j})=\lambda_{j}$ for $1\leq j\leq k$ and $\chi_{k}(\zeta_{1},\ldots,\zeta_{k})=w$.
\end{theorem}

The condition $w\leq \min_{1\leq j\leq k} \lambda_{j}$
is necessary in view of \eqref{eq:nigeria}. It would be nice to have some additional relation
between $\chi_{k}$ and $\omega_{k}$ included. In Theorem~\ref{interv}, which treats the special case
of the Veronese curve, a connection to the constants $\lambda_{k}$ will be given provided the parameter is at least $2$.
We emphasize that Theorem~\ref{lamblemm} answers Problem~\ref{que1}.  

\begin{corollary}
The spectrum of $\chi_{k}$ equals $[1,\infty]$.
\end{corollary}

Now we turn towards Question~\ref{que3}. We restrict to $\underline{\zeta}$ on the Veronese curve and denote 
the exponent $\chi_{k}(\zeta)=\chi(\zeta,\zeta^{2},\ldots,\zeta^{k})$. 
Since $\chi_{k}(\zeta)\geq \lambda_{k}(\zeta)$, from~\cite[Lemma~1]{bug} we infer
\begin{equation} \label{eq:galeich}
\chi_{k}(\zeta)\geq \frac{\lambda_{1}(\zeta)-k+1}{k}.
\end{equation}
For large parameters $\lambda_{1}(\zeta)$ and special choices of $\zeta$, very similarly 
constructed as in the proof of~\cite[Theorem~1]{bug} by Bugeaud, 
we will show in Theorem~\ref{interv} that there is equality in \eqref{eq:galeich}.
The proof of this is among other things based on the fact that there cannot be two good
approximations $p/q,p^{\prime}/q^{\prime}$ to $\zeta$ with $q,q^{\prime}$ that do not differ much.
Some parts of the proof
also involve similar ideas as the proof of~\cite[Theorem~6.2]{schleimar} or~\cite[Lemma 4.10]{schleimj}.
Our main result concerning Question~\ref{que3} is the following.

\begin{theorem}  \label{interv}
Let $k\geq 1$ be an integer. 
For $\lambda\in{[2,\infty]}$ real transcendental $\zeta$ can be explicitly constructed such that
$\chi_{k}(\zeta)=\lambda_{k}(\zeta)=\lambda$. In particular,  
the spectrum of $\chi_{k}$ on the Veronese curve contains $[2,\infty]$. 
\end{theorem}

See also the remarks subsequent to the proof of Theorem~\ref{interv} that relate
Theorem~\ref{interv} and $\zeta$ constructed in the proof with classical approximation constants. 
We end by stating the natural conjecture.

\begin{conjecture}
The spectrum of $\chi_{k}$ on the Veronese curve equals $[1,\infty]$. 
\end{conjecture}

\section{Proofs of Theorem~\ref{lamblemm} and Theorem~\ref{interv}}

The proofs heavily use the theory of continued fractions. Any irrational real number has a unique
representation as $\zeta=a_{0}+1/(a_{1}+1/(a_{2}+\cdots))$ for positive integers $a_{j}$ that can 
be recursively determined. This is called the the continued fraction expansion of $\zeta$ and
we also write $\zeta=[a_{0};a_{1},a_{2},\ldots]$. The evaluation of 
any finite subword $r_{l}/s_{l}=[a_{0};a_{1},\ldots,a_{l}]$ is called convergent to $\zeta$ and satisfies
$\vert r_{l}/s_{l}-\zeta\vert\leq s_{l}^{-2}$. More precisely we have
\begin{equation} \label{eq:ketzenbruch}
\frac{a_{l+2}}{s_{l}s_{l+2}}\leq \left\vert \frac{r_{l}}{s_{l}}-\zeta\right\vert\leq \frac{1}{s_{l}s_{l+1}}.
\end{equation}
Recall also the inductive formulas $r_{l+1}=a_{l+1}r_{l}+r_{l-1}, s_{l+1}=a_{l+1}s_{l}+s_{l-1}$.
We will utilize also the following well-known result.

\begin{theorem}[Legendre] \label{lagrangia}
If for irrational $\zeta$ the inequality
\[
\vert q\zeta-p\vert \leq \frac{1}{2}q^{-1}
\]
has an integral solution $(p,q)\in{\mathbb{Z}^{2}}$ then $p/q$ is a convergent of $\zeta$ 
in the continued fraction expansion.
\end{theorem}

\begin{proof}[Proof of Theorem~\ref{lamblemm}]
First we do not take care of the $\mathbb{Q}$-linear independence condition and in the end describe
how to modify the constructions below to ensure this additional condition.
Without loss of generality $1\leq \lambda_{1}\leq \lambda_{2}\leq \cdots\leq \lambda_{k}$.  
Let
\[
\zeta_{j}=[0;1,1,\ldots,1,h_{j,1},1,1\ldots,1,h_{j,2},1,\ldots]
\]
for the positions at which the $h_{j,i}\neq 1$ are such as the values $h_{j,i}$ to be determined later.
For $i\geq 1$ denote $r_{j,i}/s_{j,i}$ the convergent $[1,\ldots,1,h_{j,i}]$. 
Observe that by elementary estimates for continued fractions related to \eqref{eq:ketzenbruch},
for any convergent $r/s$ not equal to some $r_{j,i}/s_{j,i}$ we have $\vert s\zeta_{j}-r\vert\geq (1/3)s^{-1}$. 
Hence and by Theorem~\ref{lagrangia}, for $w>1$, every large 
solution of the system \eqref{eq:referee} for $\psi(t)=t^{-(1+w)/2}$
has each $x_{j}$ an integral multiple of some $s_{j,i}$. Similarly, if $w=1$, the argument applies
with $\psi(t)=t^{-1-\epsilon}$ for every $\epsilon>0$. Hence we may restrict $x_{j}$ of the form
$s_{j,i}$.

First define $h_{1,i}$ with sufficiently large differences $h_{1,i+1}-h_{1,i}$ 
recursively in a way that 
\[
\lim_{i\to\infty} -\frac{\log\vert \zeta_{1}s_{1,i}-r_{1,i}\vert}{\log s_{1,i}}= \lambda_{1}.
\]
This is clearly possible and leads to $\zeta_{1}=\lim_{i\to\infty} r_{1,i}/s_{1,i}$ 
that satsifies $\lambda_{1}(\zeta_{1})=\lambda_{1}$. Now we choose $h_{j,i}$ of the 
remaining $\zeta_{2},\ldots,\zeta_{k}$
with the properties
\begin{equation} \label{eq:michel}
\lim_{i\to\infty} -\frac{\log\vert \zeta_{1}s_{j,i}-r_{j,i}\vert}{\log s_{j,i}}= \lambda_{j},
\end{equation}
and 
\begin{equation} \label{eq:duglas}
\lim_{i\to\infty} \frac{\log s_{j,i}}{\log s_{1,i}}= \frac{w}{\lambda_{j}}.
\end{equation}
Such a choice is again possible. To satisfy \eqref{eq:duglas}
we just have to stop reading ones in the continued fraction expansion
at the right position, which is possible since by reading only ones two successive denominators of
convergents differ by a factor at most $2$. Then to guarantee \eqref{eq:michel}
we just have to take the next partial quotient,
that is some $h_{j,i}$, of the right order.

We prove that the implied $\zeta_{j}$ have the desired properties. Observe that since 
$w\leq \lambda_{1}\leq \ldots \leq \lambda_{k}$ and the gap between $s_{1,i}$ and $s_{1,i+1}$
can be arbitrarily large, we may assume
\begin{equation} \label{eq:reiheord}
s_{1,i}>s_{2,i}>s_{3,i}\cdots >s_{k,i}, \qquad s_{k,i+1}>s_{1,i}^{\lambda_{1}}.
\end{equation}
For $X=s_{1,i}$ and $q_{j}=s_{j,i}$ for $1\leq j\leq k$ we have by construction
\[
\lim_{i\to\infty} -\frac{\log\vert \zeta_{1}s_{j,i}-r_{j,i}\vert}{\log X}=
\lim_{i\to\infty} -\frac{\log\vert \zeta_{1}s_{j,i}-r_{j,i}\vert}{\log s_{j,i}}\frac{\log s_{j,i}}{\log X}=
 \lambda_{j}\frac{w}{\lambda_{j}}=w.
\]
Hence $\chi_{k}(\zeta_{1},\ldots,\zeta_{k})\geq w$ by the definition of the constant $\chi_{k}$. 
On the other hand, we carried out above that
we have to take each $x_{j}=s_{j,i}$ for some $i$. Thus the optimal choices are given by $X=s_{j,i}$
for some $j$. But \eqref{eq:reiheord} implies $j=1$ since otherwise if $X=s_{j,i}$ for $j\neq 1$ then 
$s_{1,i}>X$ but
\[
\lim_{i\to\infty} -\frac{\log\vert \zeta_{1}s_{1,i-1}-r_{1,i-1}\vert}{\log X}<1.
\]
This would imply $\chi_{k}(\zeta_{1},\ldots,\zeta_{k})=1$. In case of $w>1$ this 
indeed gives a contradiction. It follows in fact the choices carried out are optimal and thus 
$\chi_{k}(\zeta_{1},\ldots,\zeta_{k})\leq w$, such that there is equality.
Finally, in the case $w=1$ the above construction implies $\chi_{k}(\zeta_{1},\ldots,\zeta_{k})=1$
very similarly.

Finally we carry out how to guarantee that the vector $\underline{\zeta}$ can be chosen $\mathbb{Q}$-linearly 
independent together with $\{1\}$, by a slight modification of the above construction.
In the process
we can recursively choose $\zeta_{j}$ for $1\leq j\leq k$ in turn not in the $\mathbb{Q}$-span of
$\{1,\zeta_{1},\ldots,\zeta_{j-1}\}$. First observe
that $\zeta_{1}$ must be
transcendental if $\lambda_{1}(\zeta_{1})>1$ by Roth Theorem, and otherwise the claim of the theorem
is a trivial consequence of \eqref{eq:nigeria} for any $\mathbb{Q}$-linearly independent
vector $\underline{\zeta}$ with first coordinate $\zeta_{1}$ anyway.
For the recursive step note that the span of $j-1$ numbers is countable but we have at infinitely many 
positions at least two choices of 
positions where to put $h_{j,i}$ (it follows from the proof that the positions are not completely determined but there
is some freedom). Pigeon hole principle implies there must be uncountably many choices for $\zeta_{j}$ and
repeating this argument we obtain uncountably many vectors that
have $\mathbb{Q}$-linearly independent coordinates. 
\end{proof}

Now we turn towards the proof of Theorem~\ref{interv}. It needs some preperation.
First recall Minkowski's second lattice point Theorem~\cite{minkowski} 
asserts that for a lattice $\Lambda$ in $\mathbb{R}^{k}$ with determinant $\det \Lambda$ 
and a central-convex body $K\subseteq \mathbb{R}^{n}$ of $n$-dimensional volume $\rm{vol}(K)$, the product 
of the successive minima $t_{1},\ldots,t_{n}$ of $K$ relative to $\Lambda$ are bounded by
\[
\frac{2^{k}}{k!}\frac{\det \Lambda}{\rm{vol}(K)}\leq t_{1}t_{2}\cdots t_{k}\leq 2^{k}\frac{\det \Lambda}{\rm{vol}(K)}.
\] 
Applied in dimension $2$ and for the lattice 
$\Lambda:=\{x+\zeta y: x,y\in{\mathbb{Z}}\}$
and the $0$-symmetric convex body $K_{Q}:=\{ -Q\leq x\leq Q, -1/(2Q)\leq y\leq 1/(2Q)\}$ it yields the following.

\begin{theorem}[Minkowski] \label{thmmin}
Let $\zeta$ be a real number. Then for any parameter $Q>1$ the system
\begin{equation}  \label{eq:minkow}
\vert q\vert \leq Q, \qquad \vert \zeta q-p\vert\leq \frac{1}{2Q}
\end{equation}
cannot have two linearly independent integral solution pairs $(p,q)$. 
\end{theorem}

Moreover, we
need some facts on continued fractions which can be found in~\cite{perron}.

\begin{theorem} \label{grundsatzfrage}
For irrational $\zeta$ and every convergent $p/q$ of $\zeta$ in lowest terms we have
\[
\vert q\zeta-p\vert \leq q^{-1}.
\]
More generally, for any parameter $Q>1$ the system
\[
1\leq q\leq Q, \qquad \vert q\zeta-p\vert \leq Q^{-1}
\]
has a solution $(p,q)$ with $p/q$ a convergent of $\zeta$.
\end{theorem}

Call $q\in{\mathbb{N}}$ a {\em best approximation} of $\zeta$ if 
$\Vert q\zeta\Vert=\min_{1\leq q^{\prime}\leq q} \Vert q^{\prime}\zeta\Vert$.
As $q\to\infty$ this induces a sequence of best approximations (that uniquely determines $\zeta$). 
The following connection to the continued fraction expansion of $\zeta$ is well-known.

\begin{lemma}[Lagrange] \label{brachlemma}
The sequence of best approximations is induced by the sequence of convergents to $\zeta$. More precisely,
the $j$-th element of the sequence is the denominator of the $j$-th convergent to $\zeta$.
\end{lemma}

The next Proposition is in fact also well-known.
However, we give a proof based on Theorem~\ref{grundsatzfrage}, Theorem~\ref{lagrangia} and the fact that
for $\zeta=[a_{0};a_{1},\ldots]$ with convergents $r_{n}/s_{n}$ we have 
$s_{n+1}=a_{n+1}s_{n}+s_{n-1}$ (where formally $s_{-2}=1, s_{-1}=0$). Observe
by Lemma~\ref{brachlemma} we have $s_{n}=q_{n}$ for $q_{n}$ the $n$-th best approximation.

\begin{proposition} \label{dilemmaprop}
Let $q_{1},q_{2},\ldots$ be the sequence of best approximations of $\zeta=[a_{0};a_{1},\cdots]$. Let
\[
\nu_{n}:=-\frac{\log \Vert q_{n}\zeta\Vert}{\log q_{n}}, \qquad \eta_{n}:=\frac{\log q_{n+1}}{\log q_{n}},
\qquad \tau_{n}:=\frac{\log (a_{n+1}q_{n})}{\log q_{n}}.
\]
Then $\eta_{n}-\nu_{n}=o(1)$ and $\eta_{n}-\tau_{n}=o(1)$ as $n\to\infty$. 
\end{proposition}

\begin{proof}
The second claim follows from the fact that
for $\zeta=[a_{0};a_{1},\ldots]$ the convergents $p_{n}/q_{n}$ satisfy the recurrence 
$q_{n+1}=a_{n+1}q_{n}+q_{n-1}$ (where formally $q_{-2}=1, q_{-1}=0$). Indeed this implies 
$a_{n+1}q_{n}\leq q_{n+1}\leq (a_{n+1}+1)q_{n}$ and further by mean value theorem
of differentiation for the logarithm function $0< \eta_{n}-\tau_{n}\leq 1/\log q_{n}$ 
which tends to $0$. For the first claim note that if $\eta_{n}-\nu_{n}>2\delta>0$ for
fixed $\delta>0$ and large $n$, there is a contradiction to Theorem~\ref{grundsatzfrage}
for the parameter $Q=q_{n}^{1+\delta}$ for large $n$. On the other hand if $\eta_{n}-\nu_{n}<-2\delta<0$,
then for the parameter $Q=q_{n}^{1-\delta}$ there would be two good approximations $p_{n}/q_{n}$
and $p_{n+1}/q_{n+1}$, contradicting the Minkowski~Theorem~\ref{thmmin}.  
\end{proof}

Now we are finally ready to prove Theorem~\ref{interv}.

\begin{proof} [Proof of Theorem~\ref{interv}]
We may restrict to $k\geq 2$ since for $k=1$ clearly $\lambda_{1}(\zeta)=\omega_{1}(\zeta)$ for all $\zeta$
and the claim follows even for $\lambda\in{[1,\infty]}$ either by elementary constructions with continued fractions
or $\zeta=\sum_{n\geq 1} 2^{-a_{n}}$ with $a_{n}=\lfloor (1+\lambda)^{n}\rfloor$, 
see~\cite{bug2} for the latter.                                                     

Let $k\geq 2$ and $\lambda\in{[2,\infty]}$. We define
the continued fraction expansion of suitable $\zeta$ recursively similar to~\cite{bug}. 
Write $\zeta=[a_{0};a_{1},a_{2},\ldots]$ and $(r_{n}/s_{n})_{n\geq 0}$ the sequence of convergents as above.
Let $a_{0}=0, a_{1}=1, a_{2}=2$ such that $r_{0}/s_{0}=0, r_{1}/s_{1}=1, r_{2}/s_{2}=2/3$,
and recursively define $a_{j+1}=\lceil s_{j}^{k\lambda+k-2}\rceil$ for $j\geq 2$.
By Proposition~\ref{dilemmaprop} we have 
\begin{equation} \label{eq:hacekdicht}
\lim_{n\to\infty} -\frac{\log \vert s_{n}\zeta-r_{n}\vert}{\log s_{n}}= k\lambda+k-1.
\end{equation}
Hence Lemma~\ref{brachlemma}
implies $\lambda_{1}(\zeta)=k\lambda+k-1$ (see also~\cite{bug}). Since $\lambda>1$,
by~\cite[Corollary~1.9]{schlei} we conclude $\lambda_{k}(\zeta)=\lambda$. 
In particular $\chi_{k}(\zeta)\geq \lambda$.                    
It remains to be proved that $\chi_{k}(\zeta)\leq \lambda$.

To show this estimate, we partition the positive real numbers in successive intervals, and in each interval
give an asymptotic upper bounded at most $\max\{2,\lambda\}=\lambda$ for the $1$-dimensional constant 
$\lambda_{1}$ of some $\zeta^{i}$.
Since trivially for every parameter $Q$ the optimal exponent in the system \eqref{eq:referee} 
restricted to $q\in{[1,Q]}$ is bounded by the minimum of the related $1$-dimensional constants in this intervals
(parametrized version of right hand side of \eqref{eq:nigeria}), 
this indeed implies the upper bound $\lambda$ for $\chi_{k}(\zeta)$. 

Let $n\geq 1$ be a large integer.  
Denote $r_{m,j}/s_{m,j}$ the $m$-th convergent of $\zeta^{j}$, such that $r_{m,1}=r_{m}$ and $s_{m,1}=s_{m}$.
Observe that using the identity $A^{j}-B^{j}=(A-B)(A^{j-1}+\cdots+B^{j-1})$ and
$r_{m,j}\asymp_{\zeta} s_{m,j}$,  from \eqref{eq:hacekdicht} we obtain
\begin{equation} \label{eq:rotzpippn}
\vert s_{n}^{j}\zeta^{j}-r_{n}^{j}\vert \asymp_{\zeta} 
s_{n}^{j-1}\vert s_{n}\zeta-r_{n}\vert =s_{n}^{-k\lambda-k+j+o(1)}=(s_{n}^{j})^{-(k\lambda+k-j+o(1))/j}, 
\qquad 1\leq j\leq k.
\end{equation}
Since 
\[
\frac{k\lambda+k-j}{j}\geq \frac{k\lambda+k-k}{k}=\lambda>1, \qquad 1\leq j\leq k,
\]
by Legendre Theorem~\ref{lagrangia}, for $1\leq j\leq k$ the 
fraction $r_{n}^{j}/s_{n}^{j}$ is a convergent of $\zeta^{j}$ 
if we have chosen $n$ sufficiently large. 
Hence we may write $r_{n}^{j}/s_{n}^{j}=r_{m,j}/s_{m,j}$
where every $m=m(n,j)$ depends on $n$ and $j$ (for simplicity we write only $m$. For $j=1$ 
we will identify $m$ with $n$ such that we simply have $s_{m,1}=s_{n}$ or $m(n,1)=n$.)
Moreover \eqref{eq:rotzpippn} and Proposition~\ref{dilemmaprop} imply
\begin{equation} \label{eq:torfir}
s_{m+1,j}= s_{m,1}^{k\lambda+k-j+o(1)}= s_{n}^{k\lambda+k-j+o(1)}, \qquad 1\leq j\leq k, \quad n\to\infty.
\end{equation}
In particular $s_{n+1}=s_{n}^{k\lambda+k-1+o(1)}$ as $n\to\infty$ and 
\begin{equation} \label{eq:ordnungjo}
s_{m,1}<s_{m,2}<\cdots <s_{m,k}<s_{m+1,k}<s_{m+1,k-1}\cdots <s_{m+1,1}.
\end{equation}
We partition the interval $[s_{n},s_{n+1})=[s_{m,1},s_{m+1,1})$ in the successive pairwise disjoint intervals
\[
[s_{m,1},s_{m+1,1})=[s_{m,1},s_{m,2})\cup
\ldots\cup [s_{m,k},s_{m+1,k})\cup [s_{m+1,k},s_{m+1,k-1})\cup \ldots\cup [s_{m+1,2},s_{m+1,1}).
\]
We will prove for $Q$ in each such interval separately the upper bound $\lambda$ for the expression 
\[
\min_{1\leq j\leq k} -\frac{\log \vert \zeta^{j}q_{j}-p_{j}\vert}{\log Q}
\]
with $1\leq q_{j}\leq Q$ for $1\leq j\leq k$. Assuming this is true, since $n$ was arbitrary 
and $[s_{1},\infty)$ is obviously the disjoint union of the intervals $[s_{n},s_{n+1})=[s_{m,1},s_{m+1,1})$
over $n\geq 1$, we have that $\lambda$ is the uniform upper bound 
for $\chi_{k}(\zeta)$ as desired. For the following proof of this fact keep in mind that by construction 
and Lagrange~Lemma~\ref{brachlemma}, for
any $1\leq j\leq k$ and $Q$ in the
interval $[s_{m,j},s_{m+1,j})$, for $\zeta^{j}$ the optimal approximation in the system \eqref{eq:referee} 
with $X=Q$ is attained for $q_{j}=s_{m,j}=s_{n}^{j}$ (and $p_{j}=r_{n}^{j}$).

We start with the somehow distinguished middle interval $Q\in{[s_{m,k},s_{m+1,k})}$. 
We show that in this interval $\zeta^{k}$ cannot be approximated too well by fractions.
Indeed, the optimal choices $Q=s_{n}^{k}$ and $p_{k}=r_{n}^{k}$ and $q_{k}=s_{n}^{k}$
with \eqref{eq:rotzpippn} and \eqref{eq:torfir} for $j=k$ lead to
\[
\min_{1\leq j\leq k} -\frac{\log \vert \zeta^{j}q_{j}-p_{j}\vert}{\log Q}\leq
-\frac{\log \vert \zeta^{k}q_{k}-p_{k}\vert}{\log Q}= \frac{k\lambda}{k}+o(1) =\lambda+o(1)
\]
as $n\to\infty$. 
The claim follows for these intervals $Q\in{[s_{m,k},s_{m+1,k})}$. 

Next consider the intervals $Q\in{[s_{m+1,i+1},s_{m+1,i})}=:J_{m,i}$ for $1\leq i\leq k-1$. 
We show that for $\zeta^{i}$ there is no too good rational approximation. 
First observe that $J_{m,i}\subseteq [s_{m,i},s_{m+1,i})$ in view of \eqref{eq:ordnungjo}.
Hence the optimal approximation choices $(p_{i},q_{i})$ in the system \eqref{eq:referee}
with $1\leq q_{i}\leq Q\in{J_{m,i}}$
are given by $p_{i}=r_{n}^{i}$ and $q_{i}=s_{n}^{i}$.
The estimate $Q\geq s_{m+1,i+1}$
together with \eqref{eq:rotzpippn} and \eqref{eq:torfir} for $j=i$ lead to
\[
\min_{1\leq j\leq k} -\frac{\log \vert \zeta^{j}q_{j}-p_{j}\vert}{\log Q}\leq
-\frac{\log \vert \zeta^{i}q_{i}-p_{i}\vert}{\log Q}\leq \frac{k\lambda+k-i}{k\lambda+k-i-1}+o(1)
\]
as $n\to\infty$. Since $\lambda\geq 2$ the right hand side is much smaller than $2+o(1)\leq \lambda+o(1)$ 
and the claim follows for those intervals as well. 

The intervals of the form $I_{m,i}:=[s_{m,i},s_{m,i+1})$ for $1\leq i\leq k-1$ remain. We show that for $Q$ in these 
intervals $\zeta^{i+1}$ has no too good approximations. More precisely 
for arbitrary fixed $\epsilon>0$ and $Q\in{I_{m,i}}$ with $m\geq m_{0}(\epsilon)$ sufficiently large, 
we prove that the estimate
\begin{equation} \label{eq:pippnrotz}
\vert q\zeta^{i+1}-p\vert \leq Q^{-(i+1)/i-\epsilon}
\end{equation}
has no integral solution pair $(p,q)$ with $1\leq q\leq Q$. Provided this claim holds, with $\epsilon\to 0$ 
we infer that $\chi_{k}(\zeta)$ restricted to these intervals is again bounded by $(i+1)/i\leq 2\leq \lambda$ 
and the claim is proved. Suppose \eqref{eq:pippnrotz} has a large solution. Since $Q\geq s_{m,i}=s_{n}^{i}$ we conclude
\[
\vert q\zeta^{i+1}-p\vert \leq s_{n}^{-(i+1)-i\epsilon}\leq (s_{n}^{i+1})^{-1-i\epsilon/(i+1)}
=s_{m,i+1}^{-1-i\epsilon/(i+1)}.
\]
Hence for large $m$ we have
\[
1\leq q\leq s_{m,i+1}, \qquad \vert q\zeta^{i+1}-p\vert \leq \frac{1}{2} s_{m,i+1}^{-1}.
\]
On the other hand, recall that $r_{m,i+1}/s_{m,i+1}$ is a convergent of $\zeta^{i+1}$ with good
approximation, in particular $\vert s_{m,i+1}\zeta^{i+1}-r_{m,i+1}\vert \leq (1/2) s_{m,i+1}^{-1}$. 
Clearly $(p,q)\neq (r_{m,i+1},s_{m,i+1})$ since $q<Q\leq s_{m,i+1}$ by assumption. 
Since $r_{m,i+1}/s_{m,i+1}$ is a convergent in lowest terms, more generally 
the vectors $(p,q)$ and $(r_{m,i+1},s_{m,i+1})$ must be linearly independent.
However, the existence of two linearly independent vectors with such good approximation
contradicts Minkowski's Theorem~\ref{thmmin} for $Q=s_{m,i+1}$. 
Thus the assumption was false and there cannot be a large solution of \eqref{eq:pippnrotz}. This finishes the proof.
\end{proof}

We close with some remarks on the numbers $\zeta$ constructed in the proof, 
partly concerning classical approximation constants.

\begin{remark}
The bounds for $-\log \vert \zeta^{j}s_{m,j}-r_{m,j}\vert/\log Q$
of the corresponding $\zeta^{j}$ in the intervals constructed in the proof are,
apart from $[s_{m,k},s_{m+1,k})$, by no means considered to be sharp.  
It is reasonable that the claim of Theorem~\ref{interv} for the numbers $\zeta$ constructed
within it extends to $\lambda\in{[1,\infty]}$.
\end{remark}

\begin{remark}
A similar strategy of the proof of Theorem~\ref{interv} provides bounds for the constants $\lambda_{1}(\zeta^{j})$
for the numbers $\zeta$ constructed in it. Considering each $\zeta^{i}$ in the intervals 
$[s_{m,i},s_{m+1},i)=[s_{n}^{i},s_{m+1,i})$ and $[s_{m+1,i},s_{n+1}^{i})$ separately leads,
apart from $\lambda_{1}(\zeta)=k\lambda+k-1$, with \eqref{eq:rotzpippn} and \eqref{eq:torfir} 
and Theorem~\ref{thmmin} to
\[
\frac{k\lambda+k-j}{j}\leq \lambda_{1}(\zeta^{j})\leq \max\left\{ \frac{k\lambda+k-j}{j},
\frac{j(k\lambda+k-1)}{k\lambda+k-j}\right\} 
\]
for $2\leq j\leq k$ and any parameter $\lambda\geq \max\{1,(2j-k)/k\}$ in order to guarantee
that the left expression in the maximum is also at least $1$. Clearly the arising 
bound $\max_{1\leq j\leq k} \lambda_{1}(\zeta^{j})$ (in case of $\lambda\geq k$ such that
the condition is satisfied for $1\leq j\leq k$) for $\chi_{k}(\zeta)$ is weaker 
than the one in Theorem~\ref{interv} due to the less sophisticated chosen intervals.
\end{remark}

\begin{remark}
It is shown in~\cite[Corollary~1]{bug} that for $\zeta$ as in the proof with parameter $\lambda>1$ we have
$w_{1}(\zeta)=w_{2}(\zeta)=\cdots=w_{k}(\zeta)=k\lambda+k-1$, where $w_{k}(\zeta)$ are
the classical linear form approximation constants dual to $\lambda_{k}(\zeta)$. In particular
there is equality in Khintchine's inequality $\lambda_{k}(\zeta)\leq (w_{k}(\zeta)-k+1)/k$. 
The new contribution of Theorem~\ref{interv}
is that we can even have the equalities $\lambda_{k}(\zeta)=\chi_{k}(\zeta)=(w_{k}(\zeta)-k+1)/k$
provided $w_{k}(\zeta)\geq 3k-1$ (or $\lambda_{k}(\zeta)\geq 2$). 
\end{remark}


\begin{thebibliography}{99}

\bibitem{berdod} V. I. BERNIK and M. M. DODSON. Metric Diophantine approximation on manifolds, 
Cambridge University Press (Cambridge, 1999).

\bibitem{bu} N. BUDARINA, D. DICKINSON and J. LEVESLEY. Simultaneous Diophantine approximation
on polynomial curves. {\em Mathematika}  56 (2010), 77--85.

\bibitem{bug} Y. BUGEAUD. On simultaneous rational approximation to a real numbers and its integral powers. 
{\em Ann. Inst. Fourier (Grenoble)} (6) 60 (2010), 2165--2182.  

\bibitem{bug2}  Y. BUGEAUD. Diophantine approximation and Cantor sets. {\em Math. Ann.} 341 (2008), 677--684. 

\bibitem{buglau} Y. BUGEAUD and M. LAURENT; Exponents of Diophantine approximation and Sturmian continued fractions.
{\em Ann. Inst. Fourier (Grenoble)} 55 (2005), 773--804.

\bibitem{drutu} C. DRU\c{T}U. Diophantine approximation on rational quadrics. {\em Math. Ann.} 333, 405--470 (2005).

\bibitem{jarnik} V. JARN\'IK. \"Uber die simultanen Diophantische Approximationen. 
{\em Math. Z.} 33 (1931), 505--543.

\bibitem{khint} Y. A. KHINTCHINE. Zur metrischen Theorie der diophantischen Approximationen.
{\em Math. Z.} 24 (1926), 706--714.

\bibitem{kumar} K. KUMAR, R. THANGADURAI and M. WALDSCHMIDT. Liouville numbers and Schanuel's
conjecture, {\em Arch. Math.} 102 (2014), 59--70.

\bibitem{schleimar} D. MARQUES and J. SCHLEISCHITZ. On a problem posed by Mahler. {\em to appear in
J. Aust. Math. Soc., arXiv: 1501.02731.}                                                                

\bibitem{minkowski} H. MINKOWSKI. Geometrie der Zahlen. {\em Teubner}, Leipzig (1910).

\bibitem{perron} O. PERRON. Lehre von den Kettenbr\"uchen. {\em Teubner} (1913).

\bibitem{rieger} G.J. RIEGER. \"Uber die L\"osbarkeit von Gleichungssystemen 
durch Liouville--Zahlen. {\em Arch. Math.}, 26, no. 1 (1975), p. 40--43.  

\bibitem{schlei} J. SCHLEISCHITZ. On the spectrum of Diophantine approximation constants. 
{\em Mathematika} 62 Issue 1, p. 79--100 (2016).                                            

\bibitem{schleimj} J. SCHLEISCHITZ. Generalizations of a result of Jarn\'ik on simultaneous approximation.
{\em to appear in Mosc. J. Comb. Number T., arXiv: 1410.6697.}

\bibitem{schwarz} W. SCHWARZ. Liouville--Zahlen und der Satz von Baire. 
{\em Math.-Phys. Semesterber.} 24 (1977), 84--87.

\bibitem{sprindzuk} V.G. SPRIND\^ZUK. A proof of Mahler's conjecture on the measure of the set of S-numbers, 
{\em Izv. Akad. Nauk SSSR Ser.Mat.} 29 (1965), 379--436. 
English translation in: {\em Amer. Math. Soc. Transl.}  51 (1966), 215--272.

\end{thebibliography}
\end{document}